\documentclass[11pt]{amsart}
\usepackage{amsmath}
\usepackage{amssymb}
\usepackage{amsfonts}
\usepackage{latexsym}
\usepackage{amscd}
\usepackage{tikz} 
\usetikzlibrary{matrix,arrows,decorations.pathmorphing}
\usepackage[mathscr]{euscript}
\usepackage{xy} \xyoption{all}

\newcommand{\uloopr}[1]{\ar@'{@+{[0,0]+(-4,5)}@+{[0,0]+(0,10)}@+{[0,0] +(4,5)}}^{#1}}
\newcommand{\uloopd}[1]{\ar@'{@+{[0,0]+(5,4)}@+{[0,0]+(10,0)}@+{[0,0]+ (5,-4)}}^{#1}}
\newcommand{\dloopr}[1]{\ar@'{@+{[0,0]+(-4,-5)}@+{[0,0]+(0,-10)}@+{[0, 0]+(4,-5)}}_{#1}}
\newcommand{\dloopd}[1]{\ar@'{@+{[0,0]+(-5,4)}@+{[0,0]+(-10,0)}@+{[0,0 ]+(-5,-4)}}_{#1}}

\newcommand{\luloop}[1]{\ar@'{@+{[0,0]+(-8,2)}@+{[0,0]+(-10,10)}@+{[0, 0]+(2,2)}}^{#1}}

\newtheorem{lem}{Lemma}[section]

\newtheorem{theor}[lem]{Theorem}
\newtheorem{prop}[lem]{Proposition}
\theoremstyle{definition}
\newtheorem{defi}[lem]{Definition}

\newtheorem{exem}[lem]{Example}

\newtheorem{rema}[lem]{Remark}

\newcommand{\nc}{\newcommand}
\nc{\fin}{\mathrm{fin}}
\nc{\idn}{\mathrm{id}}
\nc{\todd}{\mathrm{td}}
\nc{\iden}{\mathrm{id}}
\nc{\op}{\mathrm{op}}
\nc{\spect }{\mathrm{Spec\ \ }}
\nc{\logg}{\mathrm{log \ }}
\nc{\set}{\mathrm{set}}
\nc{\pr}{\mathrm{Pr}}
\nc{\chern}{\mathrm{ch}}
\nc{\degg}{\mathrm{deg}}
\nc{\gr}{\mathrm{Gr }}
\nc{\spec}{\mathrm{Spec \ }}
\nc{\specf}{\mathrm{Spec}}
\nc{\Hom}{\mathrm{Hom}}
\nc{\range}{\mathrm{range \  }}
\nc{\kernel}{\mathrm{ker\ }}
\nc{\rk}{\mathrm{rank}}
\nc{\gl}{\mathrm{GL}}
\nc{\expo}{\mathrm{exp \ }}
\nc{\modulo}{\mathrm{mod \ }}
\nc{\homo}{\mathrm{Hom}}
\nc{\inte}{\mathrm{int}}
\nc{\oct}{\mathrm{Oct}}
\nc{\dist}{\mathrm{dist}}
\nc{\modd}{\mathrm{mod}}
\nc{\pic}{\mathrm{Pic}}
\begin{document}

\title[On the algebraic K-theory of $\widehat{\spec {\mathbb Z}^N}$]{On the algebraic K-theory of $\widehat{\spec {\mathbb Z}^N}$}%
\author{Stella Anevski}

\thanks{{\em Supported by the Danish National Research Foundation through the Centre for Symmetry and Deformation.}}
\date{}

\maketitle
\begin{abstract}
In his thesis, N. Durov develops a theory of algebraic geometry in which schemes are locally determined by {\em commutative algebraic monads}. In this setting, one is able to construct the Arakelov geometric compactification $\widehat{\spec{\mathbb Z}}$ of the spectrum of the ring of integers in a purely algebraic fashion. More precisely, the object $\widehat{\spec{\mathbb Z}}$ arises as the limit of a certain projective system of generalized schemes. We study the constituents $\widehat{\spec {\mathbb Z}^N}$, of this projective system, and compute their algebraic K-theory.
\end{abstract}
\section*{Introduction} 
\noindent Consider the category of locally ringed spaces ${\mathcal L}{\mathcal R}{\mathcal S}$. The objects of ${\mathcal L}{\mathcal R}{\mathcal S}$ are pairs $(X,{\mathcal O}_X)$, consisting of a topological space $X$ and a sheaf ${\mathcal O}_X$ of commutative rings with unit on $X$. The stalk ${\mathcal O}_{X,x}$ of ${\mathcal O}_X$ at any given point $x\in X$ is required to be a local ring. A morphism $(X,{\mathcal O}_X)\to (Y,{\mathcal O}_Y)$ in ${\mathcal L}{\mathcal R}{\mathcal S}$ is a pair $(f,f^\sharp)$, consisting of a continuous map $X\to Y$ and a morphism $f^\sharp: {\mathcal O}_Y\to f_{\ast}{\mathcal O}_X$ of sheaves on $Y$, such that the induced homomorphism $f_p^\sharp:{\mathcal O}_{Y,f(x)}\to {\mathcal O}_{X,x}$ is a local homomorphism of local rings.\\
\\
Let ${\mathcal Rings}$ be the category of commutative rings with unit and consider the functor of global sections
\begin{eqnarray*}
\Gamma: {\mathcal L}{\mathcal R}{\mathcal S}&\to & {\mathcal Rings},\\
(X,{\mathcal O}_X)&\mapsto &{\mathcal O}_X(X).
\end{eqnarray*}
There exists a functor $\specf: {\mathcal Rings}\to {\mathcal L}{\mathcal R}{\mathcal S}$, such that
\begin{eqnarray*}
\Hom_{{\mathcal L}{\mathcal R}{\mathcal S}}(X, \spec R)\simeq \Hom_{\mathcal Rings}(R, \Gamma (X,{\mathcal O}_X)),
\end{eqnarray*}
for any choice of $X\in {\mathcal L}{\mathcal R}{\mathcal S}$ and $R\in {\mathcal Rings}$. An {\em affine scheme} is a locally ringed space of the form $\spec R$, for some ring $R\in {\mathcal Rings}$. A {\em scheme} is a locally ringed space admitting an open cover of affine schemes.\\
 \\
 The scheme $\spec R$ has the following explicit description: The points of the underlying topological space of $\spec R$ are in bijection with the prime ideals of the ring $R$ and a base for the topology is given by the open sets
 \begin{eqnarray*}
 D_r=\{x\in \spec R; r\notin x\},
 \end{eqnarray*}
 for $r\in R$. The structure sheaf ${\mathcal O}_{\spec R}$ is determined by its values on the basis sets:
 \begin{eqnarray*}
 {\mathcal O}_{\spec R}(D_r)=R[r^{-1}].
 \end{eqnarray*}
\newpage
\noindent
 In this paper, we study generalized schemes in the sense of N. Durov (cf. \cite{durov}). A generalized scheme is locally determined by a {\em commutative algebraic monad} - a certain type of endofunctor on the category of sets. In sections \ref{gencom} and \ref{gensch}, we summarize the basics of this extended version of algebraic geometry. In particular, the construction of the prime spectrum of a generalized ring will be in complete parallel with the classical case discussed above. \\
  \\
 One of the motivations for considering an extension of classical algebraic geometry comes from Arakelov geometry. Inspired by the analogy between number fields and function fields of curves, one aims to construct an object $\widehat{\spec {\mathbb Z}}$, serving as a "compactification" of the spectrum of the integers. This object is supposed to be in analogy with projective space ${\mathbb P}^1_{k}$ over a field.  However, the stalk at infinity of $\widehat{\spec {\mathbb Z}}$ would have to be the set
 \begin{eqnarray*} \label{stalk}
 {\mathcal O}_{\widehat{\spec {\mathbb Z}},\infty}=\{\alpha\in {\mathbb R}; |\alpha|\leq 1\},
 \end{eqnarray*}
 which is a (multiplicative) commutative monoid, but not a ring. This means that $\widehat{\spec {\mathbb Z}}$ lacks a description in the framework of locally ringed spaces.\\
  \\
 Changing the setting from locally ringed spaces to structured spaces given locally by commutative algebraic monads, the object $\widehat{\spec {\mathbb Z}}$ admits a purely algebraic construction, as is shown in \S {\bf 7} of \cite{durov}. Section  \ref{algk} is devoted to the study of the constituents $\widehat{\spec {\mathbb Z}^N}$ of the projective system with limit $\widehat{\spec {\mathbb Z}}$. We describe the construction of $\widehat{\spec {\mathbb Z}^N}$, for $N>1$, and compute its algebraic K-theory. Our main result is:\\
  \\
 {\bf Theorem \ref{ktheory}}
     \begin{eqnarray*}
  K^0(\widehat{\spec {\mathbb Z}^N})={\mathbb Z}\times \pic(\widehat{\spec {\mathbb Z}^N})={\mathbb Z}\times \log {\mathbb Z}[N^{-1}]^{\ast}_{+}.
  \end{eqnarray*}
   \\
Note the similarity with Durov's computation (\cite{durov} Theorem {\bf 10.5.22}):
  \begin{eqnarray*}
  K^0(\widehat{\spec {\mathbb Z}})={\mathbb Z}\times \pic(\widehat{\spec {\mathbb Z}})={\mathbb Z}\times \log {\mathbb Q}^{\ast}_{+}.
  \end{eqnarray*}
  Indeed, the fundamentals of the proof are the same, and the only new result we need is:\\
   \\
  {\bf Theorem \ref{mainth}}\\
  Let $A_N$ be the generalized ring
  \begin{eqnarray*}
  {\bf n}\mapsto \{(\lambda_1,...,\lambda_n); \lambda_i\in {\mathbb Z}[N^{-1}] \ \ \sum_{i=1}^n |\lambda_i|\leq 1\}.
  \end{eqnarray*}
  Then any finitely generated projective module over $A_N$ is free.
\section{Generalized rings} \label{gencom}
\noindent In this section, we introduce our notion of generalized rings. We want the definition to include both commutative monoids and commutative rings with unit as special cases. In particular the object ${\mathcal O}_{\widehat{\spec {\mathbb Z}},\infty}$, i.e. the stalk at infinity of $\widehat{\spec {\mathbb Z}}$, should exist as a generalized ring. We begin in great generality, by considering monads on any category.
\newpage
\noindent
\begin{defi}
Let ${\mathcal C}$ be a category. A {\em monad} on ${\mathcal C}$ is a triple $M=(M,\mu,\epsilon)$, consisting of an endofunctor $M:{\mathcal C}\to {\mathcal C}$ and two natural transformations $\mu:M\circ M\to M$ and $\epsilon:\idn_{\mathcal C}\to M$, such that the following diagrams commute:
\begin{eqnarray*}
\begin{tikzpicture} \matrix(m)[matrix of math nodes, row sep=2.6em, column sep=2.8em, text height=1.5ex, text depth=0.25ex] {M\circ M\circ M(X) &M\circ M(X)\\ M\circ M(X)&M(X)\\}; \path[->,font=\scriptsize,>=angle 90] (m-1-1) edge node[auto] {$M(\mu_X)$} (m-1-2)
edge node[auto] {$\mu_{M(X)}$} (m-2-1) (m-1-2) edge node[auto] {$\mu_X$} (m-2-2) (m-2-1) edge node[auto] {$\mu_X$} (m-2-2); \end{tikzpicture}
\begin{tikzpicture}[>=angle 90] \matrix(a)[matrix of math nodes, row sep=3em, column sep=2.5em, text height=1.5ex, text depth=0.25ex] {\idn_{\mathcal C}\circ M(X)&M\circ M(X)&M\circ\idn_{\mathcal C}(X)\\
&M(X)\\}; \path[->](a-1-1) edge node[above]{$\epsilon_{M(X)}$} (a-1-2); \path[->](a-1-1) edge node[below left]{$\iden_{M(X)}$}(a-2-2); \path[->](a-1-2) edge node[right]{$\mu_X$}(a-2-2); \path[->](a-1-3) edge node[above]{$M(\epsilon_X)$}(a-1-2); \path[->](a-1-3) edge node[below right]{$\iden_{M(X)}$}(a-2-2);
 \end{tikzpicture}
 \end{eqnarray*}
A {\em morphism} of monads $\phi:M_1\to M_2$ is a natural transformation of the underlying endofunctors, such that
\begin{eqnarray*}
&&\phi\circ \epsilon_{M_1}=\epsilon_{M_2}  \mbox{  and}\\
&&\phi\circ \mu_{M_1}=\mu_{M_2}\circ (\phi \cdot \phi),
\end{eqnarray*}
for any morphism $f:X\to Y$ in ${\mathcal C}$. 
\end{defi}
\begin{defi} \label{moddef}
A {\em module} over a monad $M$ is a pair $X=(X,\alpha)$, consisting of an object $X$ of ${\mathcal C}$ and a morphism $\alpha:M(X)\to X$, such that $\alpha\circ\mu_X=\alpha\circ M(\alpha)$ and $\alpha\circ \epsilon_X=\idn_X$. A {\em morphism} $f:(X,\alpha_X)\to (Y,\alpha_Y)$ of $M$-modules is an element $f\in \Hom_{\mathcal C}(X,Y)$, such that $f\circ\alpha_X=\alpha_Y\circ M(f)$.
\end{defi}
\noindent We do not want our generalized rings to be unmanageble. For this reason, we shall restrict our attention to monads on the category of sets satisfying the following finiteness condition.
\begin{defi}
An endofunctor $M$ on the category of sets is {\em algebraic} if
\begin{eqnarray*}
M(\varinjlim_i X_i)\simeq \varinjlim_i M(X_i),
\end{eqnarray*}
for any filtered family of sets $\{X_i\}_i$. A monad is {\em algebraic} if its underlying endofunctor is algebraic.
\end{defi}
\noindent Since every set can be described as a filtered limit of its finite subsets, an algebraic monad needs only be defined on the finite sets ${\bf n}=\{1,...,n\}$, $n\in {\mathbb N}$. 
\begin{exem} \label{examples} (Algebraic monads on the category of sets)\\
$(i)$ Every ring $R$ defines a monad $M_R$ on the category of sets. We map a given set $X$ to the set of functions $X\to R$ having finite support in $X$:
\begin{eqnarray*}
M_R(X)=\Hom^{fin}(X,R).
\end{eqnarray*}
$(ii)$ Let
\begin{eqnarray*}
\lVert\cdot\rVert:{\mathbb R}^n&\to &{\mathbb R},\\
(\lambda_1,...,\lambda_n)&\mapsto &\sum_{i=1}^n |\lambda_i|,
\end{eqnarray*}
be the $L^1$-norm. Consider the monad ${\mathbb Z}_{\infty}$, defined via
\begin{eqnarray*}
{\mathbb Z}_{\infty}({\bf n})=\{\lambda \in {\mathbb R}^n; \lVert \lambda\rVert \leq 1\}.
\end{eqnarray*}
Then the stalk at infinity ${\mathcal O}_{\widehat{\spec {\mathbb Z}},\infty}$ of $\widehat{\spec {\mathbb Z}}$ is the monoid ${\mathbb Z}_{\infty}({\bf 1})$.\\ 
$(iii)$ For any natural number $N>1$, we define a monad $A_N$ by letting
\begin{eqnarray*}
A_N({\bf n})=\{\lambda\in {\mathbb Z}[N^{-1}]^n; \lVert\lambda\rVert\leq 1\}.
\end{eqnarray*}
This particular monad will play a prominent role in the sequel.
\end{exem}
\noindent Finally, we would like to have analogs of classical theorems from algebraic geometry at our disposal. For this reason we will need to exclude monads arising from non-commutative rings, which is done by means of the following notion of commutativity. It is designed in such a way that the monad $M_R$ in example \ref{examples} $(i)$ is commutative if and only if the ring $R$ is commutative.
\begin{defi} Let $M$ be an algebraic monad. The elements $t\in M({\bf n})$ and $s\in M({\bf m})$ {\em commute} if for every $M$-module $X$ and every family $\{x_{ij}\}\subset X$, $1\leq i\leq n$, $1\leq j\leq m$�,
\begin{eqnarray*}
t(s(x_{11},....,x_{1m}),...,s(x_{n1},...,x_{nm}))=s(t(x_{11},...,x_{1n}),...,t(x_{m1},...,x_{mn})). 
\end{eqnarray*}
The monad $M$ is {\em commutative} if all elements of $\coprod_{n\geq 0}M({\bf n})$ commute.
\end{defi}
\begin{defi}
A {\em generalized ring} is a commutative algebraic monad on the category of sets. A {\em morphism} of generalized rings is a morphism of the underlying monads. We denote the category of generalized rings by ${\mathcal G}{\mathcal Rings}$.
\end{defi}
\noindent We see that if we add an assumption that the ring $R$ in example \ref{examples} $(i)$ is commutative, then all monads in the example are generalized rings.
\subsection{Modules over generalized rings} 
We have already defined modules over monads (cf. Definition \ref{moddef}). By a {\em module over a generalized ring}, we simply mean a module over the underlying monad. \\
 \\
 In the case of the monad $M_R$ arising from a commutative ring with unit as in example \ref{examples} $(i)$, the notions $M_R$-module and $R$-module coincide. Indeed, let $X$ be a set, and for each $x\in X$, let $\chi_x$ be the characteristic function of $x$, that is
\begin{eqnarray*}
x&\mapsto &1,\\
y&\mapsto &0, \mbox{ if $y\neq x$}.
\end{eqnarray*}
Now assume that $(X,\alpha)$ is a module over $M_R$. Defining
\begin{eqnarray*}
x+y:=\alpha(\chi_x + \chi_y), \mbox{ and}\\
\lambda x:= \alpha (\lambda \chi_x),
\end{eqnarray*}
we obtain an $R$-module. Conversely, given an $R$-module $X$, the map
\begin{eqnarray*}
\lambda_1 \chi_{x_1}+...+\lambda_n \chi_{x_n}\mapsto \lambda_1 x_1+...+\lambda_n x_n,
\end{eqnarray*}
defines a $M_R$-module structure on $X$. We recover the underlying set of $R$ as $M_R({\bf 1})$ and the multiplication by regarding $R$ as a left $M_R$-module.\\
 \\
 Next, we consider some properties of modules.
\begin{defi}
Let $M$ be a monad. An $M$-module $P$ is {\em projective} if any surjection $\pi: X\to P$ of $M$-modules admits a section $\sigma:P\to X$. An $M$-module $F$ is {\em free} if there exists an $n\in {\mathbb N}$ such that  $\homo_M(F, X)\simeq X^n$ for any $M$-module $X$. This entails $F\simeq M({\bf n})$ (cf. \cite{durov} {\bf 0.4.10}). A projective module is {\em finitely generated} if it admits a surjection from a free module.
\end{defi}
\noindent When $R$ is a commutative ring, the usual notions {\em projective}, {\em free} and {\em finitely generated} coincides with the ones just defined under the correspondence between $R$-modules and $M_R$-modules described above.\\
 \newpage
 \noindent {\bf Example \ref{examples}} (continued) \label{pidth}\\
$(i)$ Recall that if $R$ is a principal ideal domain, then every finitely generated projective $R$-module is free (cf. \cite{lang} III: Theorem $7.1$). Examples of principal ideal domains are fields, polynomial rings over fields, and the ring ${\mathbb Z}$ of integers. In particular, every finitely generated projective module over the monad $M_{\mathbb Z}$ is free.\\
$(ii)$ N. Durov shows that any finitely generated projective module over ${\mathbb Z}_{\infty}$ is free (cf. \cite{durov} \S {\bf 10.4}). More precisely, a finitely generated projective ${\mathbb Z}_{\infty}$-module is an octahedron
\begin{eqnarray*}
{\mathbb Z}_{\infty}({\bf n})=\{\lambda\in {\mathbb R}^n; \lVert \lambda \rVert \leq 1\},
\end{eqnarray*}
for some $n\in {\mathbb N}$. \\
$(iii)$ In \S \ref{algk}, we show that any finitely generated projective module over $A_N$ is free (cf. theorem \ref{mainth}).
\section{Generalized schemes} \label{gensch}
\noindent Let $M$ be a generalized ring, and note that the set $M({\bf 1})$ has the structure of a commutative monoid. This allows us to consider multiplicative systems in $M({\bf 1})$.
\begin{defi}
An {\em ideal} in $M$ is an $M$-submodule of $M({\bf 1})$. An ideal $I$ is {\em prime} if $M({\bf 1})\setminus I$ is a multiplicative system. An ideal $I$ is {\em maximal} if it is maximal with respect to the partial order defined by inclusion of proper ideals of $M$.
\end{defi}
\noindent Given a multiplicative system $S\subset M({\bf 1})$, we construct the {\em localization} of $M$ with respect to $S$:
\begin{eqnarray*}
M[S^{-1}]({\bf n})=\{(a,s)\in M({\bf n})\times S\}/ \sim,
\end{eqnarray*}
where $(a,s)\sim (b,t)$ iff there exists an $u\in S$ such that $uta=usb$. The localization is initial in the category of pairs $(B,\rho)$, such that $B$ is a generalized ring and $\rho:M\to B$ is a morphism such that all elements of $\rho_1(S)\subset B({\bf 1})$ are invertible in $B({\bf 1})$ (cf. \cite{durov} {\bf 6.1.2} and {\bf 6.1.7}).\\
 \\
We equip the set $\spec M$ of prime ideals of $M$ with the {\em Zariski topology} by picking a basis of open sets consisting of the sets
\begin{eqnarray*}
D_m=\{x\in \spec M; m\notin x\},
\end{eqnarray*}
for $m\in M({\bf 1})$.\\
 \\
 Then we equip the topological space $\spec M$ with a sheaf ${\mathcal O}_{\spec M}$ of generalized rings:
\begin{eqnarray*}
 {\mathcal O}_{\spec M}(D_m)=M[ m^{-1}].
\end{eqnarray*}
\begin{defi} 
Let $M$ be a generalized ring. The {\em prime spectrum} of $M$ is the pair 
\begin{eqnarray*}
\spec M=(\spec M, {\mathcal O}_{\spec M}).
\end{eqnarray*} 
\end{defi}
\noindent This construction should be compared with the classical construction of the prime spectrum of a commutative ring with unit.
\newpage
\noindent {\bf Example \ref{examples}} (continued)\\
$(i)$ If $R$ is a commutative ring with unit, then the prime spectrum of $M_R$ coincides with the prime spectrum of $R$. In the sequel, we will therefore denote $M_R$ by $R$ and $\spec M_R$ by $\spec R$.\\
$(ii)$ ${\mathbb Z}_{\infty}$ is a {\em local generalized ring}, i.e. it has a unique maximal ideal. This is the ideal
\begin{eqnarray*}
\{\alpha\in {\mathbb R}; |\alpha|< 1\},
\end{eqnarray*}
which is also the only prime ideal of ${\mathbb Z}_{\infty}$.\\
$(iii)$ Consider the generalized ring $A_N$. By \cite{durov} {\bf 7.1.4}, the underlying topological space of $\spec A_N$ is
\begin{eqnarray*}
\{(0), (p),..., \infty\}_{p\nmid N},
\end{eqnarray*}
where $\infty=\{\alpha\in {\mathbb Z}[N^{-1}]; |\alpha|<1\}$. A non-empty subset of $\spec A_N$ is open if and only if it contains $\infty$ and has finite complement.
\subsection{Generalized locally ringed spaces} \label{glrs}
Consider a morphism $(f,f^{\sharp}):(X,{\mathcal O}_X)\to (Y,{\mathcal O}_Y)$ in the category ${\mathcal L}{\mathcal R}{\mathcal S}$ of locally ringed spaces. Requiring the induced homomorphism of stalks $f_p^\sharp:{\mathcal O}_{Y,f(x)}\to {\mathcal O}_{X,x}$, to be a local homomorphism of local rings for any $x\in X$ is equivalent to the requirement that whenever $U\subset X$ and $V\subset Y$ are affine open subschemes such that $f(U)\subset V$, the restricted morphism
\begin{eqnarray*}
f_{U,V}:U\to V,
\end{eqnarray*}
is induced by a ring homomorphism ${\mathcal O}_X(V)\to {\mathcal O}_X(U)$. This leads to the following definition of the category of generalized locally ringed spaces.
\begin{defi}
A {\em generalized locally ringed space} is a pair $(X,{\mathcal O}_X)$, consisting of a topological space $X$ and a sheaf of generalized rings ${\mathcal O}_X$ on $X$. The stalk ${\mathcal O}_{X,x}$ of $X$ at any given point $x\in X$ is required to be a local generalized ring. A {\em morphism} $(X,{\mathcal O}_X)\to (Y,{\mathcal O}_Y)$ of generalized locally ringed spaces is a pair $(f,f^\sharp)$, consisting of a continuous map $f:X\to Y$ and a morphism $f^\sharp: {\mathcal O}_Y\to f_{\ast}{\mathcal O}_X$ of sheaves on $Y$. Moreover, we require that for any open prime spectra $\spec M_1\subset X$ and $\spec M_2\subset Y$, such that $f(\spec M_1)\subset \spec M_2$, the restricted morphism $f:\spec M_1\to \spec M_2$ is induced by a homomorphism $M_2\to M_1$ of generalized rings. We denote the category of generalized locally ringed spaces by ${\mathcal G}{\mathcal L}{\mathcal R}{\mathcal S}$.
\end{defi}
\begin{defi}
A {\em generalized scheme} is a generalized locally ringed space which admits an open cover by prime spectra of generalized rings.
\end{defi}
\noindent An immediate consequence of this definition is that morphisms of affine schemes $\spec M_1\to \spec M_2$ are in one-to-one correspondence with generalized ring homomorphisms $M_2\to M_1$:
\begin{eqnarray*}
\Hom_{{\mathcal G}{\mathcal L}{\mathcal R}{\mathcal S}}(\spec M_1,\spec M_2)\simeq \Hom_{{\mathcal G}{\mathcal Rings}}(M_2,M_1).
\end{eqnarray*}
More generally (cf. \cite{durov} {\bf 6.5.2}), if we let $\Gamma:{\mathcal G}{\mathcal L}{\mathcal R}{\mathcal S}\to {\mathcal G}{\mathcal Rings}$ be the functor of global sections, then the functor $\specf: {\mathcal G}{\mathcal Rings}\to {\mathcal G}{\mathcal L}{\mathcal R}{\mathcal S}$, satisfies
\begin{eqnarray*}
\Hom_{{\mathcal G}{\mathcal L}{\mathcal R}{\mathcal S}}(X, \spec R)\simeq \Hom_{{\mathcal G}{\mathcal Rings}}(R, \Gamma (X,{\mathcal O}_X)),
\end{eqnarray*}
in complete analogy to the classical situation described in the introduction.
\newpage
\noindent {\bf Example \ref{examples}} $(iii)$ (continued)\\
By \cite{durov} {\bf 7.1.2}, the ring ${\mathbb Z}[N^{-1}]$ is the localization of the generalized ring $A_N$ with respect to the multiplicative system generated by $N^{-1}$. Since ${\mathbb Z}[N^{-1}]$ is also the localization of $\mathbb Z$ with respect to the multiplicative system generated by $N$, this allows us to glue $\spec A_N$ and $\spec {\mathbb Z}$ along the common open subset $\spec{\mathbb Z}[N^{-1}]$. The resulting generalized scheme will be denoted by $\widehat{\spec {\mathbb Z}^N}$:
\begin{eqnarray*}
\widehat{\spec {\mathbb Z}^N}=\spec A_N\coprod_{\spec \mathbb Z[N^{-1}]}\spec {\mathbb Z}.
\end{eqnarray*}
As a set $\widehat{\spec {\mathbb Z}^N}\simeq \spec {\mathbb Z}\cup \{\infty\}$. The point $(0)$ is generic, and hence contained in any non-empty open subset of $\widehat{\spec {\mathbb Z}^N}$. The points $\infty$ and $(p)$ for $p|N$ are closed. The remaining points $(p)$ are tangled with $\infty$, in the sense that $\overline{(p)}=\{p,\infty\}$ whenever $p\nmid N$. A non-empty subset $U\subset \widehat{\spec {\mathbb Z}^N}$ is open if and only if it contains $(0)$, its complement is finite, and if it either doesn't contain $\infty$ or contains all $p\nmid N$ (cf. \cite{durov} {\bf 7.1.6}).\\
\subsection{Vector bundles}
Given a generalized locally ringed space $(X,{\mathcal O}_X)$, we can consider the category ${\mathcal O}_X-\modd$ of sheaves of ${\mathcal O}_X$-modules. Objects are sheaves ${\mathcal E}$ of modules on $X$ such that ${\mathcal E}(U)$ is an ${\mathcal O}_X(U)$-module for any open set $U\subset X$. Furthermore, for each inclusion of open sets $V\subset U$, the restriction morphism ${\mathcal E}(U)\to {\mathcal E}(V)$ is required to be compatible with the module structures via the morphism ${\mathcal O}_X(U)\to {\mathcal O}_X(V)$ of generalized rings. A {\em morphism} ${\mathcal E}\to {\mathcal F}$ of sheaves of ${\mathcal O}_X$-modules is a morphism of sheaves, such that for each open set $U\subset X$, the map ${\mathcal E}(U)\to {\mathcal F}(U)$ is a morphism of ${\mathcal O}_X(U)$-modules.
\begin{defi}
A {\em vector bundle} over a generalized locally ringed space $(X,{\mathcal O}_X)$ is an ${\mathcal O}_X$-module ${\mathcal E}$, such that ${\mathcal E}_{|U_i}\simeq {\mathcal O}_{X|U_i}({\bf n}_i)$ as an ${\mathcal O}_{X|U_i}$-module, for some open cover $U_i$ of $X$ and some integers $n_i\geq 0$. If all $n_i$ can be chosen to have the same value $n$, ${\mathcal E}$ is said to have {\em rank $n$}. A {\em line bundle} over $X$ is a vector bundle of rank $1$.
\end{defi}
\begin{rema}
If ${\mathcal E}$ is a vector bundle over a generalized scheme $X$, then the set $\Gamma(U,{\mathcal E})$ of sections of ${\mathcal E}$ over $U$ is a finitely generated projective module for any affine open $U\subset X$. The converse is not always true (cf. \cite{durov} {\bf 6.5.29}).
\end{rema}
\begin{defi}
A {\em morphism} $u:{\mathcal E}' \to {\mathcal E}$ of vector bundles over a generalized locally ringed space $(X,{\mathcal O}_X)$ is a morphism of the underlying ${\mathcal O}_X$-modules. The morphism $u$ is a {\em cofibration} if it can be locally presented as a retract of a standard embedding $e:{\mathcal E}'_{|U_i}\to {\mathcal E}'_{|U_i}\oplus {\mathcal O}_{X|U_i}({\bf n})$, i.e. if there exist morphisms 
\begin{eqnarray*}
j:&&{\mathcal E}_{|U_i}\to {\mathcal E}'_{|U_i}\oplus {\mathcal O}_{X|U_i}({\bf n})  \mbox{ and}\\
q:&&{\mathcal E}'_{|U_i}\oplus {\mathcal O}_{X|U_i}({\bf n})\to {\mathcal E}_{|U_i}, 
\end{eqnarray*}
of ${\mathcal O}_{X|U_i}$-modules such that $q\circ j=\iden$, $e=j\circ u$ and $u=q\circ e$.
\end{defi}
\noindent If the category ${\mathcal O}_X-\modd$ is pointed, we define the {\em cofiber} ${\mathcal E}''$ of a cofibration $u:{\mathcal E}' \to {\mathcal E}$ to be the pushout of $u$ along the zero morphism in ${\mathcal O}_X-\modd$:
\begin{eqnarray*}
\begin{tikzpicture} \matrix(m)[matrix of math nodes, row sep=2.6em, column sep=2.8em, text height=1.5ex, text depth=0.25ex] {{\mathcal E}' &{\mathcal E}\\ 0&{\mathcal E}''\\}; \path[->,font=\scriptsize,>=angle 90] (m-1-1) edge node[auto] {$u$} (m-1-2)
edge node[auto] {$$} (m-2-1) (m-1-2) edge node[auto] {$$} (m-2-2) (m-2-1) edge node[auto] {$0$} (m-2-2); \end{tikzpicture}
\end{eqnarray*}
\begin{defi} Let $(X, {\mathcal O}_X)$ be a generalized ringed space. The {\em Grothendieck group} $K^{0}(X)$ is generated by isomorphism classes $[{\mathcal E}]$ of vector bundles over $X$ modulo relations
\begin{eqnarray*}
&(i)& [\emptyset]=0, \\
&(ii)& \mbox{$[{\mathcal E}]=[{\mathcal E}']+[{\mathcal E}'']$, whenever there is a cofibration }\\
&&\mbox{${\mathcal E}'\to {\mathcal E}$ in ${\mathcal O}_X-\modd$ with cofiber ${\mathcal E}''$.}\\
\end{eqnarray*}
The {\em Picard group} $\pic(X)$ is the set of isomorphism classes of line bundles over $X$ with group operation given by the tensor product (cf. \cite{durov} {\bf 5.3.5} and {\bf 6.5.30}).
\end{defi}
\begin{lem} \label{structurelemma}
Let $(X,{\mathcal O}_X)$ be a generalized ringed space and consider the ring $R={\mathbb Z}\times \pic(X)$, with multiplication defined by $(m,x)\cdot(n,y)=(mn,nx+my)$. Assume that the map
\begin{eqnarray*}
c_1:\pic(X)&\to& K^0(X),\\
{\mathcal L}&\mapsto& [{\mathcal L}]-1,
\end{eqnarray*}
is a homomorphism of abelian groups. Then the map
\begin{eqnarray*}
\varphi: R&\to& K^0(X),\\
(n,{\mathcal L})&\mapsto& n-1+[{\mathcal L}],
\end{eqnarray*}
is a homomorphism of rings. If $K^0(X)$ is generated by line bundles, then $\varphi$ is surjective. If both ${\mathbb Z}\to K^0(X)$ and $c_1$ are injective, then $\varphi$ is injective.
\end{lem}
\begin{proof}
See \cite{durov} {\bf 10.5.21}.
\end{proof}
\section{Algebraic K-theory of $\widehat{\spec {\mathbb Z}^N}$} \label{algk}
\noindent Recall the construction of $\widehat{\spec {\mathbb Z}^N}$ from example \ref{examples} in \S \ref{glrs}. In this section, we compute the algebraic K-theory of $\widehat{\spec {\mathbb Z}^N}$. This is done using the methods of \cite{durov} \S {\bf 10.5}, where the algebraic K-theory of $\widehat{\spec {\mathbb Z}}$ is computed. The result which allows us to apply these methods in our case is theorem \ref{mainth}, stating that any finitely generated projective $A_N$-module is free.
\subsection{$A_N$-modules}
Let $P$ be a finitely generated projective $A_N$-module and let $\pi:A_N({\bf n})\to P$ be the surjective map from the free $A_N$-module on $n$ generators. We assume that $n$ is minimal. Choose a section $\sigma:P\to A_N({\bf n})$ and put $a=\sigma\circ\pi$. Then $a^2=a$ and $a$ is given by a matrix $A=(a_{ij})$, with entries $a_{ij}\in A_N({\bf 1})$. For $1\leq j\leq n$, let $u_j$ be the image of the standard generator $e_j$ of $A_N({\bf n})$. Then
\begin{eqnarray*}
\lVert u_j\rVert=\sum_{i=1}^{n} |a_{ij}|\leq 1, \mbox{ for all $1\leq j\leq n$},
\end{eqnarray*}
since $u_j=a(e_j)\in A_N({\bf n})$.
\begin{lem} \label{norm}
We have $\lVert u_i\rVert = 1$, for all $1\leq i\leq n$.
\end{lem}
\begin{proof}
Let $I=\{i\in {\bf n}; \lVert u_i\rVert=1\}$. For any vector $x=(x_1,...,x_n)\in {\mathbb R}^n$, we have
\begin{eqnarray*}
\lVert Ax\rVert=\sum_{i=1}^{n} \left| \sum_{j=1}^{n} a_{ij} x_j \right| \leq \sum_{i=1}^{n}\sum_{j=1}^{n}|a_{ij}||x_j|\leq \sum_{j=1}^{n}|x_j|=\lVert x\rVert.
\end{eqnarray*}
In particular, if $\lVert Ax\rVert=\lVert x\rVert$, then
\begin{eqnarray*}
\sum_{i=1}^{n}\sum_{j=1}^{n}|a_{ij}||x_j|= \sum_{j=1}^{n}|x_j|,
\end{eqnarray*}
which is possible only if for all $j$, either $x_j=0$ or $\sum_i|a_{ij}|=1$, i.e. $j\in I$. Since $\lVert Au_i\rVert=\lVert u_i\rVert$, $u_i$ belongs to the $A_N({\bf 1})$-span of $\{u_i\}_{i\in I}$, so $I={\bf n}$, by minimality of $n$.
\end{proof}
\noindent Note that if $R=(r_{ij})$ is the matrix with entries $r_{ij}=|a_{ij}|$, then $R^2=R$. Indeed, since $A^2=A$, we have $a_{ik}=\sum_j a_{ij} a_{jk}$, so
\begin{eqnarray*}
r_{ik}=|a_{ik}|\leq \sum_{j=1}^{n}|a_{ij}||a_{jk}|=\sum_{j=1}^{n}r_{ij}r_{jk}.
\end{eqnarray*}
We also have
\begin{eqnarray*}
1=\sum_{i=1}^{n}r_{ik}\leq \sum_{i=1}^{n}\sum_{j=1}^{n} r_{ij} r_{jk}=\sum_{j=1}^{n}r_{jk}=1,
\end{eqnarray*}
and hence $r_{ik}=\sum_j r_{ij}r_{jk}$.
\begin{prop} \label{columns}
If $r_{ij}>0$, for all $1\leq i,j\leq n$, then all columns of $R$ are equal. 
\end{prop}
\begin{proof}
Let $x=(x_1,...,x_n)^t$ be the transpose of the $i$:th row of $R$. We show that all components of $x$ are equal. Put $m=\min_i x_i$. Replacing $x$ by $x-m(1,...,1)^t$, we may assume that all $x_i\geq 0$ and that at least one $x_k=0$. By lemma \ref{norm} and since $R^2=R$, we have $R^t x=x$, so
\begin{eqnarray*}
x_k=\sum_{i=1}^{n} r_{ik} x_i.
\end{eqnarray*}
Since all $r_{ik}$'s are strictly positive, this implies $x_i=0$, for all $1\leq i\leq n$.
\end{proof}
\noindent Now we show that if $r_{ij}>0$, for all $1\leq i,j\leq n$, then we may assume that $A=R$. First, note that any equality $|x_1+...+x_n|=|x_1|+...+|x_n|$ implies that $x_1+...+x_n$ has the same sign as $x_i$, for all $1\leq i\leq n$. In particular
\begin{eqnarray*}
\left| \sum_{j=1}^n a_{ij} a_{jk}\right| =|a_{ik}|=r_{ik}=\sum_{j=1}^n r_{ij} r_{jk}=\sum_{j=1}^n |a_{ij} a_{jk}|,
\end{eqnarray*}
so $a_{ij} a_{jk}$ has the same sign as $\sum_j a_{ij} a_{jk}=a_{ik}$. This means that $a_{ii}>0$, for all $1\leq i\leq n$. Put $\epsilon_i=\frac{a_{1i}}{|a_{1i}|}$, and let $B=(\epsilon_i a_{ij} \epsilon_j^{-1})$. Then $B$ defines the same finitely generated projective $A_N$-module as $A$. Since $\epsilon_1=1$, and $\epsilon_1 a_{1j} \epsilon_j^{-1}=|a_{1j}|=p_{1j}$, we may assume that $a_{1i}>0$, for all $1\leq i\leq n$. But then $a_{ij}>0$, for all $1\leq i,j\leq n$, since $a_{1i} a_{ij}$ has the same sign as $a_{1j}$. \\
By proposition \ref{columns}, all columns of $A$ are equal and therefore, by minimality of $n$, we have $A= \left( 1 \right)$.
\begin{theor}
If $r_{ij}=0$ for some $i,j\in {\bf n}$, then there exist finitely generated projective $A_N$-modules $P'$ and $P''$, such that
\begin{eqnarray*}
P=P'\oplus P''.
\end{eqnarray*}
\end{theor}
\begin{proof}
Fix $j\in {\bf n}$ and consider
\begin{eqnarray*}
S_j=\{i; r_{ij}>0\}.
\end{eqnarray*}
Choose $i_0$ such that $|S_{i_0}|$ is minimal and put $S=S_{i_0}$. Note that $i_0\in S$, since otherwise $u_{i_0}$ would lie in the $A_N({\bf 1})$-span of $\{e_j\}_{j\neq i_0}$, contradicting minimality of $n$. Since $r_{ij}=\sum_k r_{ik}r_{kj}$, $k\in S_j$ implies $S_k\subset S_j$. By minimality of $|S|$, $i,j\in S$ implies $S_i=S_j=S$. This means that
\begin{eqnarray*}
i,j\in S \Rightarrow r_{ij}>0  \mbox{ and}\\
i\notin S, j\in S \Rightarrow r_{ij}=0.
\end{eqnarray*}
In other words, we may reorder indices so that $S={\bf k}$, for some $k<n$. Then the matrix $A$ is block-triangular
\begin{eqnarray*}
A= \left( \begin{array}{ccc}
A' & C \\
0 & A'' \end{array} \right)
\end{eqnarray*}
with $A'^2=A'$ and $A''^2=A''$. Let $P'$ be the finitely generated projective $A_N$-module defined by $A'$. By \cite{durov} {\bf 10.2.12}, there exist maps
\begin{eqnarray*}
\sigma':P'\oplus A_N({\bf n}-{\bf k})\to P  \mbox{ and}\\
j:P\to P'\oplus A_N({\bf n}-{\bf k}),
\end{eqnarray*}
such that $\sigma' \circ j=\iden_P$. Put $q=j\circ \sigma'$, and let $q'$ be the composite
\begin{eqnarray*}
A_N({\bf n}-{\bf k})\to P'\oplus A_N({\bf n}-{\bf k}) \stackrel{q}{\to}P'\oplus A_N({\bf n}-{\bf k})\to A_N({\bf n}-{\bf k}).
\end{eqnarray*}
Then $q^2=q$ and $q'^2=q'$. Now we show that $q=\iden_{P'}\oplus q'$. Let $\{e_i\}$ be the standard basis of $A_N({\bf n}-{\bf k})$, and put $m_i=q(e_i)$, for $1\leq i\leq n-k$. Then
\begin{eqnarray*}
m_i=\lambda_i v_i+\mu_i w_i,
\end{eqnarray*}
for some $v_i\in P'$, $w_i\in A_N({\bf n}-{\bf k})$, and with $|\lambda_i|+|\mu_i|\leq 1$. Let $m_i'$ be the projection of $m_i$ on $A_N({\bf n}-{\bf k})$, that is $m_i'=q'(e_i)=\mu_i w_i$. 
Since $q'^2=q'$, we have $\lVert q'(m_i')\rVert=\lVert m_i'\rVert$ and hence
\begin{eqnarray*}
\lVert q'(w_i)\rVert=\lVert w_i\rVert.
\end{eqnarray*}
By the same argument as in lemma \ref{norm}, this means that $(w_i)_j=0$ unless $\lVert m_i'\rVert=1$. Setting $I=\{i; \lVert m_i'\rVert=1\}$, we can write
\begin{eqnarray*}
w_i &=& \sum_{j\in I}(w_i)_j e_j,\\
q(w_i) &=& \sum_{j\in I} (w_i)_j m_j,
\end{eqnarray*}
and hence
\begin{eqnarray*}
m_i=q(m_i)=\lambda_i v_i + \mu_i q(w_i).
\end{eqnarray*}
This means that $m_i$ lies in the $A_N$-submodule of $P'\oplus A_N({\bf n}-{\bf k})$ generated by $P'$ and $\{m_j\}_{j\in I}$. Minimality of $n-k$ now implies that $|\mu_i| \lVert w_i\rVert=\lVert m_i'\rVert=1$, for $1\leq i\leq n-k$. Since $w_i\in A_N({\bf n}-{\bf k})$, we have $\lVert w_i\rVert\leq 1$. Since also $|\lambda_i|+|\mu_i|\leq 1$, we conclude that $\lambda_i=0$, i.e. $m_i\in A_N({\bf n}-{\bf k})$. This means that $q=\iden_{P'}\oplus q'$, so that $P=P'\oplus P''$ for a finitely generated projective $A_N$-module $P''$.
\end{proof}
\noindent By induction on the rank of projective modules, we conclude that $P$ is free, being a direct sum of free modules. In summary, we have proved the following.
\begin{theor} \label{mainth}
If $P$ is a finitely generated projective $A_N$-module, then $P$ is free.
\end{theor}
\subsection{Vector bundles over $\widehat{\spec {\mathbb Z}^N}$}
According to \cite{durov} {\bf 7.1.19}, example \ref{examples} $(i)$, and theorem \ref{mainth}, the category of vector bundles over $\widehat{\spec {\mathbb Z}^N}$ is equivalent with the category of triples $(E_{\mathbb Z}, E_N, \theta_N)$, where $E_{\mathbb Z}$ is a free ${\mathbb Z}$-module, $E_N$ a free $A_N$-module and $\theta_N$ a ${\mathbb Z}[N^{-1}]$-module isomorphism $E_{\mathbb Z}[N^{-1}]\to E_N[N^{-1}]$. It follows from these considerations that any vector bundle over $\widehat{\spec {\mathbb Z}^N}$ has a well-defined rank.\\
  \\
Given a positive integer $r$, the vector bundles of rank $r$ admit an alternative description as a certain double coset of matrices. Namely, let $E$ be a free ${\mathbb Z}[N^{-1}]$-module and choose a basis $(f_i)_{1\leq i\leq r}$ of $E_{\mathbb Z}\subset E$ and a basis $(e_i)_{1\leq i\leq r}$ of  $E_N\subset E$. Both being bases of $E$, they are related to each other by means of a matrix $A=(a_{ij})\in \gl_r({\mathbb Z}[N^{-1}])$:
\begin{eqnarray*}
e_i=\sum_{j=1}^r a_{ij} f_j, \mbox{ \ \ } \ a_{ij}\in {\mathbb Z}[N^{-1}].
\end{eqnarray*}
If we choose another basis for $E_{\mathbb Z}$, it is related to $(f_i)_{1\leq i\leq r}$ by means of a matrix $B=(b_{ij})\in \gl_r({\mathbb Z})$:
\begin{eqnarray*}
f_i=\sum_{j=1}^r b_{ij}f'_j.
\end{eqnarray*}
 Similarly, if we replace $(e_i)_{1\leq i\leq r}$ by another basis of $E_N$, the two bases are related to each other by means of a matrix $C=(c_{ij})$ in the group $\oct_r$ of symmetries of the $r$-octahedron:
\begin{eqnarray*}
e_i'=\sum_{j=1}^r c_{ij} e_j.
\end{eqnarray*} 
Hence, multiplying $A$ from the right by matrices from $\gl_r({\mathbb Z})$ and from the left by matrices from $\oct_r$ does not change the corresponding vector bundle. Conversely, if $A$ and $A'$ define isomorphic vector bundles, we may assume that they arise from different choices of bases in the same modules $E_{\mathbb Z}$ and $E_N$. Hence $A'=CAB$ for some $C\in \oct_r$ and some $B\in \gl_r({\mathbb Z})$. We have just shown:
\begin{prop} \label{picprop}
Isomorphism classes of vector bundles of rank $r$ over $\widehat{\spec {\mathbb Z}^N}$ are in one-to-one correspondence with double cosets
\begin{eqnarray*}
\oct_r \backslash \gl_r({\mathbb Z}[N^{-1}])/\gl_r({\mathbb Z}).
\end{eqnarray*}
In particular, 
\begin{eqnarray*}
\pic(\widehat{\spec {\mathbb Z}^N})=\{\pm 1\} \backslash {\mathbb Z}[N^{-1}]^{\ast}/\{\pm 1\}=\log {\mathbb Z}[N^{-1}]^{\ast}_{+}.
\end{eqnarray*}
We denote the line bundle corresponding to an element $\log \lambda\in \log {\mathbb Z}[N^{-1}]^{\ast}_{+}$ by ${\mathcal O}(\log \lambda)$.
\end{prop}
\begin{rema} The notation $\log {\mathbb Z}[N^{-1}]^{\ast}_{+}$ arises from the convention that $\pic(\widehat{\spec {\mathbb Z}^N})$ should be written additively, while ${\mathbb Z}[N^{-1}]^{\ast}_{+}$ is written multiplicatively. 
\end{rema}
\newpage
\noindent Now let $u:{\mathcal E}'\to {\mathcal E}$ be a cofibration of vector bundles with cofiber ${\mathcal E}''$. Choose a ${\mathbb Z}[N^{-1}]$-module $E$. Set $E_{\mathbb Z}'=\Gamma(\spec {\mathbb Z}, {\mathcal E}')$, $E_{\mathbb Z}=\Gamma(\spec {\mathbb Z}, {\mathcal E})$, $E_N'=\Gamma(\spec A_N, {\mathcal E}')$ and $E_N=\Gamma(\spec A_N, {\mathcal E})$. Let $(f_i')_{1\leq\i\leq r}$ and $(f_i'')_{1\leq i\leq s}$ be bases of $E_{\mathbb Z}'$ and $E_{\mathbb Z}/E_{\mathbb Z}'$ in $E$, respectively. If we let $f_{r+i}$ be any lifts of $f_i''$, $1\leq i\leq s$, then
\begin{eqnarray*}
f_i=f_i' \mbox{\ \ if $1\leq i\leq r$   and $f_i$ if $r<i\leq r+s$,}
\end{eqnarray*}
is a basis of $E_{\mathbb Z}$. By theorem \ref{mainth}, $E_N'\to E_N$ is isomorphic to $A_N({\bf r})\to A_N({\bf r}+{\bf s})$, so we may choose a basis $(e_i)_{1\leq i\leq r+s}$ of $E_N$, such that its first $r$ elements constitute a basis of $E_N'$ and such that the images of the remaining $s$ elements constitute a basis of the cofiber $E_N''$ of $E_N'\to E_N$. Let $A\in \gl_{r+s}({\mathbb Z}[N^{-1}])$ be the matrix relating $(e_i)_{1\leq i\leq r+s}$ to $(f_i)_{1\leq i\leq r+s}$ and define $A'\in \gl_{r}({\mathbb Z}[N^{-1}])$ and $A''\in \gl_{s}({\mathbb Z}[N^{-1}])$ similarly. Then $A$, $A'$ and $A''$ describe ${\mathcal E}$, ${\mathcal E}'$ and ${\mathcal E}''$, respectively and by construction
\begin{eqnarray*}
A= \left( \begin{array}{ccc}
A' & 0 \\
C & A'' \end{array} \right)
\end{eqnarray*}
Conversely, if a vector bundle ${\mathcal E}$ admits a description in terms of a block triangular matrix $A$ of the above form, then we obtain a cofibration ${\mathcal E}'\to {\mathcal E}$ of vector bundles over $\widehat{\spec {\mathbb Z}^N}$.
We have shown:
\begin{prop}
Cofibrations of vector bundles over $\widehat{\spec {\mathbb Z}^N}$ correspond to block-triangular decompositions of corresponding matrices.
\end{prop}
\noindent Given any right coset $\gl_r({\mathbb Z}[N^{-1}])/\gl_r({\mathbb Z})$, one can construct a matrix $A=(a_{ij})\in \gl_r({\mathbb Z}[N^{-1}])$ in the same coset satisfying the following conditions:
\begin{eqnarray*}
\bullet&&\mbox{$A$ is lower-triangular, i.e. $a_{ij}=0$ for $i<j$.}\\
\bullet&&\mbox{$A$ has positive diagonal elements, i.e. $a_{ii}>0$.}\\
\end{eqnarray*}
This is done in \cite{durov} {\bf 10.5.15}.
\begin{lem} \label{surjlem}
$K^0(\widehat{\spec {\mathbb Z}^N})$ is generated by line bundles.
\end{lem}
\begin{proof}
Let ${\mathcal E}$ be a vector bundle of rank $r$ and let $A$ be a matrix of the above type describing ${\mathcal E}$. Since $A$ is in particular block triangular, it corresponds to a cofibration of vector bundles. Since $A$ is in fact completely triangular, we get the equality
\begin{eqnarray*}
[{\mathcal E}]=\sum_{i=1}^{r}[{\mathcal O}(\log a_{ii})],
\end{eqnarray*}
in $K^0(\widehat{\spec {\mathbb Z}^N})$.
\end{proof}
\begin{lem} \label{injlem}
The map
\begin{eqnarray*}
c_1:\log {\mathbb Z}[N^{-1}]^{\ast}_{+}&\to& K^0(\widehat{\spec {\mathbb Z}^N}),\\
\log \lambda &\mapsto& [{\mathcal O}(\log \lambda)]-1,
\end{eqnarray*}
is an injective homomorphism of abelian groups.
\end{lem}
\begin{proof}
First we check that
\begin{eqnarray*}
[{\mathcal O}(\log \lambda + \log \mu)]-1=[{\mathcal O}(\log \lambda)]-1 + [{\mathcal O}(\log \mu)]-1,
\end{eqnarray*}
in $K^0(\widehat{\spec {\mathbb Z}^N})$. Consider the vector bundle ${\mathcal E}$ over $K^0(\widehat{\spec {\mathbb Z}^N})$ defined by the matrix
\begin{eqnarray*}
A= \left( \begin{array}{ccc}
\lambda & 0 \\
1 & \mu \end{array} \right)
\end{eqnarray*}
Then $[{\mathcal E}]=[{\mathcal O}(\log \lambda)]+ [{\mathcal O}(\log \mu)]$ in $K^0(\widehat{\spec {\mathbb Z}^N})$. Multiplying $A$ by $\left( \begin{array}{ccc}
0 & 1\\
1 & 0 \end{array} \right)\in \oct_2$ from the left, we obtain
\begin{eqnarray*}
A'= \left( \begin{array}{ccc}
1 & \mu \\
\lambda & 0 \end{array} \right)
\end{eqnarray*}
Consider the canonical form $A''=(a''_{ij})$ of $A'$ in the sense the above discussion. Since $A'$ is congruent to $A''$ modulo $\gl_2({\mathbb Z})$, the row g.c.d.'s of $A'$ and $A''$ must be equal, so $a''_{11}=gcd(1,\mu)=1$. Also, since matrices in $\gl_2({\mathbb Z})$ have determinant $\pm 1$, $a''_{22}=\det A''=\pm \det A=\pm \lambda\mu$. All numbers involved are positive, so we get $a''_{22}=\lambda\mu$. Since $A''$ also defines ${\mathcal E}$, we have $[{\mathcal E}]=[{\mathcal O}(\log \lambda+ \log \mu)]+ 1$ in $K^0(\widehat{\spec {\mathbb Z}^N})$.
To prove injectivity, we consider the map 
\begin{eqnarray*}
\degg:K^0(\widehat{\spec {\mathbb Z}^N})&\to &\log {\mathbb Z}[N^{-1}]^{\ast}_{+},
\end{eqnarray*}
which sends the class of a vector bundle ${\mathcal E}$ represented by the matrix $A$ to $\log |\det A|$. Since $\degg\circ c_1=\iden$, $c_1$ is injective.
\end{proof}
\begin{theor} \label{ktheory}
\begin{eqnarray*}
K^{0}(\widehat{\spec {\mathbb Z}^N})\simeq {\mathbb Z}\times \log {\mathbb Z}[N^{-1}]^{\ast}_{+}.
\end{eqnarray*}
\end{theor}
\begin{proof}
Recall from proposition \ref{picprop} that $\pic(\widehat{\spec {\mathbb Z}^N})={\mathbb Z}[N^{-1}]^{\ast}_{+}$ and consider the map
\begin{eqnarray*}
\varphi: {\mathbb Z}\times \log {\mathbb Z}[N^{-1}]^{\ast}_{+}&\to& K^0(X),\\
\end{eqnarray*}
from lemma \ref{structurelemma}. Combining this lemma with lemma \ref{surjlem} and lemma \ref{injlem}, we conclude that $\varphi$ is a surjective ring homomorphism. To show injectivity, it remains to show that ${\mathbb Z}\to K^{0}(\widehat{\spec {\mathbb Z}^N})$ is injective. But this follows from the existence of a map $\spec {\mathbb Q}\to \widehat{\spec {\mathbb Z}^N}$, inducing a map $K^{0}(\widehat{\spec {\mathbb Z}^N})\to K^0(\spec {\mathbb Q})={\mathbb Z}$.

\end{proof}
\section*{Acknowledgments} \noindent The author wishes to thank Ryszard Nest and Henrik Densing Petersen.

\end{document}